\documentclass[reqno,english]{amsart}

\usepackage{amsmath,amsfonts,amssymb,graphicx,amsthm,enumerate,url}
\usepackage[noadjust]{cite}
\usepackage{stmaryrd}
\usepackage{comment,paralist}
\usepackage{mathrsfs,booktabs,tabularx}
\usepackage{xifthen,xcolor,tikz,setspace}
\usetikzlibrary{decorations.pathmorphing,patterns,shapes,calc,decorations}
\usetikzlibrary{decorations.pathreplacing}
\usepackage{mathtools}

 \usepackage[letterpaper]{geometry}
\usepackage[colorinlistoftodos]{todonotes}
\usepackage[colorlinks=true]{hyperref}
\numberwithin{equation}{section}

\newcommand{\Po}{\operatorname{Po}}

\renewcommand{\epsilon}{\varepsilon}

\newcommand{\one}{\mathbf{1}}
\usepackage{color}

\newtheorem{maintheorem}{Theorem}

\newtheorem{theorem}{Theorem}[section]
\newtheorem*{theorem*}{Theorem}
\newtheorem{lemma}[theorem]{Lemma}

\newtheorem{claim}[theorem]{Claim}

\newtheorem*{observation*}{Observation}
\newtheorem{fact}[theorem]{Fact}
\newtheorem{corollary}[theorem]{Corollary}

\newtheorem{remark}[theorem]{Remark}

\theoremstyle{definition}{

\newtheorem*{definition*}{Definition}

}

\newcommand{\E}{\mathbb E}

\renewcommand{\P}{\mathbb P}
\newcommand{\F}{\mathbb F}

\newcommand{\R}{\mathbb R}
\newcommand{\Z}{\mathbb Z}
\newcommand{\cK}{\mathcal K}
\newcommand{\cT}{\mathcal T}

\DeclareMathOperator{\var}{Var}

\begin{document}

\title{Minimum weight disk triangulations and fillings
}

\author{Itai Benjamini}
\address{I.\ Benjamini\hfill\break
Department of Mathematics\\
Weizmann Institute of Science\\
Rehovot 76100, Israel.}
\email{itai.benjamini@weizmann.ac.il}

\author{Eyal Lubetzky}
\address{E.\ Lubetzky\hfill\break
Courant Institute\\ New York University\\
251 Mercer Street\\ New York, NY 10012, USA.}
\email{eyal@courant.nyu.edu}

\author{Yuval Peled}
\address{Y.\ Peled\hfill\break
Courant Institute\\ New York University\\
251 Mercer Street\\ New York, NY 10012, USA.}
\email{yuval.peled@courant.nyu.edu}

\begin{abstract}
    We study the minimum total weight of a disk triangulation using vertices out of $\{1,\ldots,n\}$, where the boundary is the triangle $(123)$ and the $\binom{n}3$ triangles have independent weights, e.g.~$\mathrm{Exp}(1)$ or~$\mathrm{U}(0,1)$.
    We show that for explicit constants $c_1,c_2>0$, this minimum is $c_1 \frac{\log n}{\sqrt n} + c_2 \frac{\log\log n}{\sqrt n} + \frac{Y_n}{\sqrt n}$ where the random variable $Y_n$ is tight, and it is attained by a  triangulation that consists of $\frac14\log n + O_{\textsc{p}}(\sqrt{\log n}) $ vertices. 
    Moreover, for disk triangulations that are canonical, in that no inner triangle contains all but~$O(1)$ of the vertices, the minimum weight has the above form with the law of $Y_n$ converging weakly to a shifted~Gumbel. In addition, we prove that, with high probability, the minimum weights of a homological filling and a homotopical filling of the cycle $(123)$ are both attained by the minimum weight disk triangulation.
\end{abstract}

\maketitle
\section{Introduction}
We consider the following question: what is the minimum possible weight of a triangulation of a disk with fixed boundary vertices $1,2,3$ using any number of inner vertices whose labels are taken from $[n]=\{1,\ldots,n\}$, in the setting where every triangle in $\binom{[n]}3$ is assigned, for instance, an independent rate-1 exponential weight? This may be viewed as a weighted version of the Linial--Meshulam~\cite{LM06} model (each triangle is present with probability $p$, independently, and one of the questions is whether there exists a disk triangulation using these), as well as a  2-dimensional simplicial complex analog of first passage percolation on a complete graph (minimum weight paths between fixed vertices after assigning independent weights to the $\binom{n}2$ edges; in our setup, one instead looks at weighted fillings of a fixed triangle); see~\S\ref{subsec:related} for more details on related works.

Our main result establishes tightness of the minimum weight around an explicit centering term. 
For a class of triangulations, which comprises a fixed proportion out of all disk triangulations, we further identify the limiting law of the centered minimum to be Gumbel. The problem becomes more subtle when one allows vertex labels to repeat, whence the triangulation corresponds to a \emph{null-hompotopy} of the boundary. Another well-studied generalization of triangulations is a \emph{homological filling}. We show that, with high probability, all three notions achieve the same minimum, unlike the situation for instance in the Linial--Meshulam model.

\subsection{Setup and main results}\label{subsec:main}
A triangulation $T$ of the cycle $(123)$ over $[n]$ is a planar graph embedded in the triangle whose boundary vertices are labeled $1,2,3$, in which every face is a triangle and all internal vertices have labels in $[n]$. We say that such a triangulation is \emph{proper} if no two vertices share a label. Given an assignment $w$ of positive weights to the $\binom n3$ triangles on $n$ vertices, the weight of a triangulation $T$ is defined as $w(T) = \sum_{x\in T}w_x$, where the sum is taken over the labeled faces of $T$. 

We consider an assignment $\{w_x\,:\;x\in\binom{[n]}3\}$ of independent random weights whose laws (which need not be identical) satisfy 
\begin{equation}\label{eq:weights-assump}
    \P(w_x\le t)=\bigg(1+o\Big(\frac{1}{\log(1/t)}\Big)\bigg)t\qquad\mbox{as $t\downarrow 0$}\,,
\end{equation}
with the main examples being the uniform $\mathrm{U}(0,1)$ and the exponential $\mathrm{Exp}(1)$ distributions. 

\begin{maintheorem}\label{mainthm:tight}
Assign independent weights $\{w_x\,:\;x\in \binom{[n]}3\}$ to the $\binom{n}3$ triangles on $n$ vertices, with laws satisfying~\eqref{eq:weights-assump}, and let $W_n$ be the minimum of $w(T)$ over all triangulations $T$ of $(123)$ over $[n]$.
Then
\begin{equation}
    \label{eq:Wn-decomposition}
    W_n = \frac{3\sqrt 3}{16}\bigg(\frac12\frac{\log n}{\sqrt n} +  \frac52\frac{\log\log n}{\sqrt n} + \frac{Y_n}{\sqrt n}\bigg)
\end{equation}
for a sequence of random variables $(Y_n)$ that is uniformly tight. In addition, the a.s.\ unique triangulation attaining $W_n$ contains $\frac14\log n + O_{\textsc p}(\sqrt{\log n})$ vertices and, with high probability, it is proper.
\end{maintheorem}

(In the above theorem we used the notation $X_n=O_{\textsc p}(f_n)$ to denote that the sequence $X_n/f_n$ is tight.)

\begin{figure}
    \centering
     \begin{tikzpicture}[vertex/.style = {inner sep=0.75pt,circle,draw,fill,label={#1}}]
 	\newcommand\rad{2}	

	\begin{scope}
  \coordinate (v1) at (90:\rad);
  \coordinate (v2) at (210:\rad);
  \coordinate (v3) at (330:\rad);
  \coordinate (x) at (-.3,0.3);
  \coordinate (x1) at (-.6,0.5);
  \coordinate (y) at (0,-0.2);
  \coordinate (z) at (.5,0.5);
  \coordinate (z1) at (.2,0.3);
  \coordinate (z2) at (0,0.15);
  \coordinate (y1) at (0,-0.6);
    \draw [thick,black] (v1.center)--(v2.center)--(v3.center)--cycle;

	\path [fill=gray!30] (x.center)--(y.center)--(z.center)--cycle;
    \draw [black] (x.center)--(y.center)--(z.center)--cycle;
	\foreach \y in {v1, v2}
		\draw [black] (x.center)--(\y.center);
		\foreach \y in {v1, v3}
		\draw [black] (z.center)--(\y.center);
		\foreach \y in {v2, v3}
		\draw [black] (y.center)--(\y.center);
		\foreach \y in {v2, v3,y}
		\draw [black] (y1.center)--(\y.center);
		\foreach \y in {v2, v1,x}
		\draw [black] (x1.center)--(\y.center);
		\foreach \y in {z1, y,x}
		\draw [black] (z2.center)--(\y.center);
		\foreach \y in {z, y,x}
		\draw [black] (z1.center)--(\y.center);
    
  	\node[vertex={above:$1$}] at (v1) {};
  	\node[vertex={below:$2$}] at (v2) {};
  	\node[vertex={below:$3$}] at (v3) {};
	\node[vertex] at (x) {};
	\node[vertex] at (y) {};
	\node[vertex] at (z) {};
	\node[vertex] at (y1) {};
	\node[vertex] at (x1) {};
	\node[vertex] at (z1) {};
	\node[vertex] at (z2) {};
	\end{scope}
	
 \end{tikzpicture}
     \caption{A triangulation of $(123)$ that is $5$-canonical but not $6$-canonical (via the shaded triangle).  
     }
    \label{fig:canon}
    \vspace{-0.1in}
\end{figure}
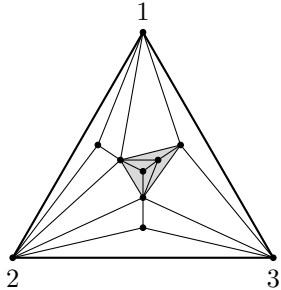

One can easily define classes of triangulations of $(123)$ where all but a fixed number of the internal vertices are confined to some inner triangle $K$, and where the minimum weights $W_n$ will give rise to different distributions of the tight random variables $Y_n$ from Theorem~\ref{mainthm:tight} (see for instance the examples in Remark~\ref{rem:other-dist}). 
Our next theorem shows that in the absence of such an inner triangle $K$, the limiting law of $Y_n$ is unique.
For $\nu\geq 1$, a triangulation $T$ of the triangle $(123)$ is  \emph{$\nu$-canonical} if for every triangle $K\neq(123)$ whose edges belong to $T$ there are at least $\nu$ internal vertices in $T$ that do not lie in the interior of~$K$ (see Fig.~\ref{fig:canon}).
When the minimization problem is restricted to proper canonical triangulations of $(123)$, we find the exact description of the limit distribution of the minimum weight, as follows.
\begin{maintheorem}\label{mainthm:gumbel}
Assign independent weights $\{w_x\,:\;x\in \binom{[n]}3\}$ to the $\binom{n}3$ triangles on $n$ vertices, with laws satisfying~\eqref{eq:weights-assump}, let $\nu=\nu(n)$ be any sequence of integers such that $\nu \leq \frac1{10}\log n$ and $\nu\to\infty$ with $n$, and
 let $W^*_n$ be the minimum of $w(T)$ over all $\nu$-canonical proper triangulations $T$ of $(123)$ over $[n]$. Letting 
\begin{equation}
    \label{eq:Wn*-decomopostion}
    W^*_{n} = \frac{3\sqrt3}{16}\bigg(\frac12 \frac{\log n}{\sqrt{n}} + \frac52 \frac{\log\log n}{\sqrt{n}} 
    + \frac{\log(\frac49\sqrt{2\pi})}{\sqrt n}
    -  \frac{Y^*_{n}}{\sqrt n}\bigg)\,,
\end{equation}
the random variable $Y^*_n$ converges weakly to a Gumbel random variable as $n\to\infty$.
\end{maintheorem}

There is an important topological distinction between proper and improper triangulations. A proper triangulation $T$, when viewed as a simplicial complex, is homeomorphic to a topological planar disk. Let us denote by $D_n$ the minimum weight over these disk triangulations of $(123)$. On the other hand, if repeated labels are allowed, it is more suitable to refer to $T$ as a {\em homotopical filling} of $(123)$ --- a null-homotopy of $(123)$ that is not necessarily homeomorphic to a disk. 

In addition, it is natural to also consider the algebraic-topology notion of a homological filling of $(123)$. Here we work over the field $\F_2$, and consider the boundary operator 
$\partial_2:\mathfrak{C}_2\to \mathfrak{C}_1$, where $\mathfrak{C}_2$ and $\mathfrak{C}_1$ are vector spaces over the field $\F_2$ spanned by $\{e_{ijk}\,:\;ijk\in\binom{[n]}3\}$ and
$\{e_{ij}\,:\;ij\in\binom{[n]}2\}$ respectively, and $\partial_2 e_{ijk}=e_{ij}+e_{ik}+e_{jk}$.
An $\F_2$-homological filling of $(123)$ is a vector $z\in \mathfrak{C}_2$ such that $\partial_2z=\partial_2e_{123}$. 
Note that even though this definition is purely algebraic, in Section \ref{sec:fillings} we  use a known geometric characterization of $\F_2$-homological fillings as triangulations of surfaces of arbitrary genus.

The weight of an $\F_2$-homological filling $z$ is defined as $\sum_{x\in\mathrm{supp}(z)}w_x$, and we denote by $F_n$ the minimum weight of an $\F_2$-homological filling of $(123)$. The boundary operator is defined such that the characteristic vector (modulo $2$) $z=z(T)$ of the faces of a triangulation $T$ of $(123)$ over $[n]$ is an $\F_2$-homological filling.
Indeed, for every appearance of a pair $ij\in\binom{[n]}2$ as an internal labeled edge of $T$, the contributions to $(\partial_2 z)_{ij}$ of the two faces of $T$ which contain the edge cancel out. Therefore, $\partial_2 z$ is equal to the characteristic vector of the edges in the boundary of $T$. This observation implies that $w(z(T))\le w(T)$. 

Consequently, the following inequality
holds point-wise for every assignment of weights $\{w_x\,:\;x\in\binom{[n]}3\}$:
\begin{equation}\label{eqn:Wn_inequality}
F_n \le W_n \le D_n.    
\end{equation}
In general, both inequalities in (\ref{eqn:Wn_inequality}) can be strict. There are classical examples of topological spaces containing a null-homotopic cycle that does not enclose a topological disk, or a null-homologous cycle that is not null-homotopic.
Moreover, there is a stark difference between the threshold probabilities for the appearances of homological and homotopical fillings in the Linial--Meshulam model (see~\S\ref{subsec:related}). We show that under independent random weights, the minimum weights of a null-homotopy and an $\F_2$-homological filling are attained by a proper disk triangulation with high probability.

\begin{maintheorem}\label{mainthm:nonplanar}
Assign independent weights $\{w_x\,:\;x\in \binom{[n]}3\}$ to the $\binom{n}3$ triangles on $n$ vertices, with laws satisfying~\eqref{eq:weights-assump}. Then, with high probability, every inclusion-minimal $\F_2$-homological filling $z$ that is not the characteristic vector of a proper triangulation of $(123)$ satisfies
\[
w(z) \ge \frac{9\sqrt{3}}{32}\frac{\log n}{\sqrt{n}}=(3-o(1))D_n.
\]
As a result, $F_n = W_n = D_n$ with high probability.
\end{maintheorem}
We stress that even though Theorem~\ref{mainthm:nonplanar} implies that with high probability every improper triangulation~$T$ is suboptimal, it does not imply that $T$ is suboptimal by a factor of $3-o(1)$. Indeed, the  $\F_2$-homological filling $z(T)$ corresponding to $T$ can be a characteristic vector of a proper triangulation (due to cancelling out of faces with repeated labels).

It will be convenient throughout the paper to work with weights  $\{w_x\,:\;x\in\binom{[n]}3\}$ that are i.i.d.\ $\mathrm{Exp}(1)$. This is enabled  by the next observation, proved by a routine coupling argument (cf., e.g.,~\cite{janson99}).
\begin{observation*}
Let $\{w_x\,:\; x\in\binom{[n]}3\}$ be independent random weights with laws satisfying~\eqref{eq:weights-assump}, and fix~$C>0$. For every $n$ there exists a coupling of these weights to $\{\tilde w_x\,:\; x\in\binom{[n]}3\}$ that are i.i.d.\ $\mathrm{U}(0,1)$ such that $\max\{ |w(T) -\tilde{w}(T)| \,:\; T\subset \binom{[n]}3\,,\,w(T)\wedge \tilde{w}(T) \geq C\frac{\log n}{\sqrt n}\} = o(1/\sqrt{n})$ with probability $1$.
\end{observation*}
To see this, let $F_x$ be the \textsc{cdf} of $w_x$ and consider its Skorokhod representation $\sup\{y \,:\; F_x(y)\leq \tilde w_x\}$. Let~$T$ be such that $w(T)\wedge \tilde{w}(T)\leq C\frac{\log n}{\sqrt n}$. Since $w_x \leq C\frac{\log n}{\sqrt n}$ for every $x\in T$, a.s.\ $\tilde w_x = F_x(w_x) = (1+o(\frac{1}{\log n}))w_x$, and summing over $x\in T$ it follows that $|w(T) - \tilde w(T)| \leq o((w(T)+\tilde{w}(T))/ \log n) = o(1/\sqrt{n})$ by assumption.  

\subsection{Related work}\label{subsec:related}
We next mention several models for which the problem studied here may be viewed as an analog or a generalization, as well as related literature on these.

\subsubsection*{First passage percolation / combinatorial optimization on the complete graph with random edge weights}

Consider the probability space where each of the $\binom n2$ edges of the complete graph on $n$ vertices is assigned an independent weight (e.g., $\mathrm{Exp}(1)$ or $\mathrm{U}(0,1)$). 
The distribution of shortest paths between fixed vertices $v_1,v_2$ (i.e., mean-field first passage percolation) has been studied in detail: Janson~\cite{janson99} showed that,
under the same assumption on the distribution of the edge weights as in~\eqref{eq:weights-assump}, this distance is distributed as
$(\log n + \Lambda_n)/n$, where $\Lambda_n$ converges weakly as $n\to\infty$ to the sum of three independent Gumbel random variables; that work further studied the number of edges in the shortest path between $v_1$ and $v_2$ (the ``hopcount''), as well as worst-case choices for $v_1$ and $v_2$. The law of the centered diameter (worst-case choices for both $v_1$ and $v_2$) was finally established by Bhamidi and van der Hofstad~\cite{BH17a}. Many flavors of this problem have been analyzed, e.g., on random graphs such as Erd\H{o}s--R\'enyi / configuration models under various degree assumptions --- see for instance the recent work~\cite{BH17b} and the references therein. 

For other related combinatorial problems on the complete graph with random edge weights (e.g., the famous $\zeta(3)$ Theorem of  Frieze~\cite{frieze85} for the minimum spanning tree, and
the random assignment problem and its $\zeta(2)$ asymptotic limit by Aldous~\cite{Aldous01}), see the comprehensive survey~\cite[\S4,\S5]{aldous04}. 

In our situation, analogously to the distance between two fixed vertices $v_1,v_2$ with random edge weights, we assign random weights to $2$-dimensional faces and Theorems~\ref{mainthm:tight} and~\ref{mainthm:gumbel} address the minimum total weight of triangulations of the fixed cycle $(123)$. Furthermore, via this analogy, the number of triangles in the triangulation achieving the minimum (addressed in Theorem~\ref{mainthm:tight}) is the counterpart of hopcount.

\subsubsection*{Minimal fillings in groups}
The area of a cycle $C$ in a simplicial complex is commonly defined as the number of $2$-faces in a minimal triangulation of whose boundary is $C$ (see, e.g.,~\cite{BHK11}).  
This terminology is motivated by the combinatorial group theoretic notion of the area of a word $w$ with respect to a group presentation --- the minimum number of $2$-cells in a diagram with boundary label $w$ (see~\cite{LS77}).
In this context, it is known that homological fillings can exhibit different asymptotics compared to homotopical fillings (see, e.g.,~\cite{ABDY13}).

Here we consider the area of a fixed cycle $(123)$ under random weights for the $2$-faces, generalizing the above definition of area to feature the total weight of the $2$-faces instead of their number. Perhaps surprisingly, Theorem~\ref{mainthm:nonplanar} shows that, in our case, the optimal homological and homotopical fillings of a fixed cycle $(123)$ coincide with high probability.

\subsubsection*{The Linial--Meshulam random $2$-dimensional simplicial complex model}
This model, denoted $Y_2(n,p)$, is a model of a random $n$-vertex $2$-dimensional simplicial complex with a full $1$-dimensional skeleton where every $2$-dimensional face appears independently with probability $p=p(n)$. Upon introducing the model, Linial and Meshulam~\cite{LM06} showed that for every fixed $\varepsilon>0$, if $p= (2+\varepsilon)\frac{\log n}n$ then $Y_2(n,p)$ is $\F_2$-homologically connected with high probability (that is, every cycle has an $\F_2$-homological filling), whereas at $p=(2-\epsilon)\frac{\log n}n$ typically there is an uncovered (isolated) edge. 
The behavior in the critical window around $p=2\frac{\log n}n$ was thereafter established by Kahle and Pittel~\cite{KP16}. On the other hand, homotopical fillings of~$(123)$ appear only at a much denser regime where $p=n^{-1/2+o(1)}$, as shown by Babson, Hoffman and Kahle~\cite{BHK11} in their study of the fundamental group of $Y_2(n,p)$. In addition, it was shown in~\cite{Luriap} that the critical probability for having a proper disk triangulation of $(123)$ in $Y_2(n,p)$ is at $p=(\frac{3\sqrt3}{16}+o(1))n^{-1/2}$, with the upper bound achieved by a triangulation that typically has at most $C \log n$ faces. However, it is not known whether this critical $p$ is also the threshold probability for simple-connectivity. See, e.g., the survey~\cite{Kahle18} for more on this model.

For the problem studied here (where faces are associated with continuous weights as opposed to Bernoulli variables), one may infer from the above results on $Y_2(n,p)$ that $n^{-1/2-\epsilon}\le W_n \le D_n\le C' \log n/\sqrt{n}$, with high probability,
 by restricting the attention to faces below a certain threshold weight. Namely, for the threshold $\mu = (\frac{3\sqrt3}{16}+\epsilon)n^{-1/2}$ we arrive at the aforementioned critical $p$ from~\cite{Luriap}, so the total weight of the triangulation would be at most $C \mu \log n $. On the other hand, $W_n\geq \mu = n^{-1/2-\epsilon}$ by~\cite{BHK11}, as otherwise we would have a triangulation where the total weight --- hence also the weight of every face --- is at most $\mu$. 
 
 Theorem~\ref{mainthm:tight} derives the correct centering terms of $D_n$ and $W_n$.   Theorem~\ref{mainthm:nonplanar} further shows that the difference between homotopical and homological fillings is not present under random independent weights, bridging the gap between $W_n$ and $D_n$. Note that the analog of this in $Y_2(n,p)$ remains a challenging open problem. 
  
  This qualitative difference between the models hinges on the following observation. On first sight, it may seem that Theorem~\ref{mainthm:nonplanar} is in conflict with the result of Linial and Meshulam on homological fillings in $Y_2(n,p)$. For instance, in the settings of Theorem~\ref{mainthm:nonplanar}, there is an $\F_2$-homological filling of $(123)$ that is supported on triangles $x$ with weight $w_x\le O(\frac{\log n}n)$ --- which is quite small relative to $W_n$. However, the weight of this filling is actually substantially larger than $W_n$ since it contains some $C n^2$ triangles with high probability. This fact was proved in \cite{ALLM13} and further studied in the work of Dotterrer, Guth and Kahle~\cite{DGK18} concerning the homological girth of $2$-dimensional complexes which inspired our proof of Theorem~\ref{mainthm:nonplanar}. 

\subsection*{Organization}
In Section~\ref{sec:proper} we prove the assertions of Theorems~\ref{mainthm:tight} and~\ref{mainthm:gumbel} when the minimum is taken only over proper triangulations. Afterwards, in Section~\ref{sec:fillings}, we prove Theorem~\ref{mainthm:nonplanar} which also completes the proof Theorem~\ref{mainthm:tight} for general triangulations.

\section{Proper triangulations below a given weight}\label{sec:proper}

In this section we examine the number of proper triangulations --- where all the vertex labels are distinct --- below a given weight $\omega_n$. (Theorem~\ref{mainthm:nonplanar}, proved in Section~\ref{sec:fillings}, shows that repeated vertex labels lead to a weight that is with high probability suboptimal.) Throughout this section, we take this target weight to be
\begin{equation}\label{eq:omega-window}
\omega_n=\omega_n(A)=\frac{\frac12 \log n + \frac52\log\log n + A}{\sqrt{\gamma n}}\qquad\mbox{for $\gamma=\frac{256}{27}$ and a fixed $A\in\R$}\,.
\end{equation}
We consider planar triangulations where the outer face is labeled $(123)$. Denote by $\cT_k$ the set of planar triangulations using $k$ unlabeled internal vertices, and by $\cT_{k,n}$ the set of planar triangulations where the $k$ internal vertices have distinct labels in $\{4,\ldots,n\}$. Further setting $\Delta_k := |\cT_k|$ (whence $|\cT_{k,n}| = \Delta_k \binom{n-3}k k! $ for every $k\leq n-3$), Tutte~\cite{Tutte62} famously showed that  
\begin{equation}\label{eq:Delta-k} \Delta_k = \frac{6(4k+1)!}{k!(3k+3)!} = \bigg(\frac1{16}\sqrt{\frac3{2\pi}}+o(1)\bigg)k^{-5/2}\gamma^{k+1}\qquad\mbox{for}\qquad\gamma=\frac{256}{27}\,.
\end{equation}
Analogously to these notations, let $\cT_k^{\nu}$ be the set of $\nu$-canonical triangulations via $k$ internal unlabeled vertices and outer face $(123)$, and let $\cT_{k,n}^{\nu}$ be the set of their counterparts with internal vertices labeled in $\{4,\ldots,n\}$ .   

Our goal is to estimate the number of triangulations $T\in\bigcup_k\cT_{k,n}$ (as well as in an appropriate interval of values for $k$) whose weight is at most the above given $\omega_n$, 
and analogously for $T\in\bigcup_k\cT_{k,n}^{\nu}$. 
Throughout the section we will use the following notation:
\begin{equation}\label{eq:Zn-def}Z_n=Z_n(\omega_n) = \big|\{T\in\bigcup_k\cT_{k,n}\,:\;w(T)\leq \omega_n\}\big|
\,,\quad
Z_n^*=Z_n^*(\omega_n) = \big|\{T\in\bigcup_k\cT_{k,n}^{\nu}\,:\;w(T)\leq \omega_n\}\big|\,,\end{equation}
as well as
\begin{equation}\label{eq:tilde-Zn-def}\widetilde Z_n=\widetilde Z_n(\omega_n) = \big|\{T\in\bigcup_{k\in\cK_a}\cT_{k,n}\,:\;w(T)\leq \omega_n\}\big|
\,,\quad
\widetilde Z_n^*= \widetilde Z_n^*(\omega_n) = \big|\{T\in\bigcup_{k\in\cK_a}\cT_{k,n}^{\nu}\,:\;w(T)\leq \omega_n\}\big|\,,\end{equation}
where
\begin{equation}\label{eq:K-def} \cK_a=[\tfrac14\log n - a \sqrt{\log n}\,,\,\tfrac14\log n + a\sqrt{\log n}]\qquad(a>0)\,.\end{equation}
With these definitions, our main result in this section is the following.
\begin{theorem}
\label{thm:poisson}
Assign i.i.d.\ $\mathrm{Exp}(1)$ weights $\{w_x\,:\;x\in\binom{[n]}3\}$ to the $\binom{n}3$ triangles on $n$ vertices. Let $\nu=\nu(n)$ be any sequence of integers such that $\nu\leq \frac{1}{10}\log n$ and $\nu\to\infty$ with $n$, fix $A\in\R$
and set $\omega_n(A)$ as in~\eqref{eq:omega-window}.
Then $Z_n(\omega_n)$ and $Z_n^*(\omega_n)$ as given in~\eqref{eq:Zn-def} satisfy
 \begin{equation*}
 \E Z_{n} \xrightarrow[n\to\infty]{} \frac83  \sqrt{\frac2\pi} e^A\qquad\mbox{and}\qquad
 Z_n^*\xrightarrow[n\to\infty]{\mathrm{d}} \Po\bigg(\frac9{4\sqrt{2\pi}} e^A\bigg)\,.
 \end{equation*}
Moreover, if $a_n$ is a sequence of integers such that $a_n = o(\sqrt{\log n})$ and $\lim_{n\to\infty}a_n=\infty$, and $\widetilde Z_n(\omega_n)$, $\widetilde Z_n^*(\omega_n)$ are as in~\eqref{eq:tilde-Zn-def} w.r.t.~$\cK_{a_n}$,
then $Z_n - \widetilde Z_n \to 0$ and $Z_n^* - \widetilde Z_n^* \to 0$ in probability as $n\to\infty$.
\end{theorem}

\subsection{Canonical and partial triangulations}
The proof of Theorem~\ref{thm:poisson} will require several combinatorial estimates on the number of $\nu$-canonical triangulations and the number of proper subsets of such triangulations with a prescribed number of internal and boundary vertices. 

\subsubsection{Enumerating canonical triangulations}
Recall that a triangulation $T\in\cT_k$ is $\nu$-canonical if it does not have an inner face $K$ containing more than $k-\nu$ internal vertices.

\begin{lemma}\label{lem:number-canonical}
Let $\nu=\nu(k)$ be a sequence such that $1 \leq \nu < k/2$ and $\nu\to\infty$ as $k\to\infty$. Then the number of $\nu$-canonical triangulations of $(123)$ with $k$ internal vertices satisfies 
\[ |\cT_k^{\nu}| = \big(\big(\tfrac34\big)^3-o(1)\big) \Delta_k \,.\]
\end{lemma}
\begin{proof}
Let $A_s$ ($s=1,\ldots,\nu$) be the set of triangulations with $k$ internal vertices where there is an internal face $K$ with exactly $k-s$ internal vertices (which thus must be unique, as $s\leq\nu<k/2$). Further let $A_s^*$ be the triangulations $T\in A_s$ attaining the maximum number of internal vertices in such a face, i.e.,
\[ A_s^* := A_s\setminus \bigcup_{t<s} A_{t}\,.\]
To enumerate $A_s^*$, we first argue that
\[ |A_s| = (2s+1) \Delta_{k-s}\Delta_s\,,\]
since triangulations in $A_s$ are in bijection with triangulations of $(123)$ via $s$ internal vertices, along with a choice of one out of the $2s+1$ faces and a triangulation of that face via $k-s$ internal vertices. 

We next argue that for every $1\leq t \leq s-1$,
\[ \frac{|A_{s-t}^* \cap A_s|}{|A_{s-t}^*|} = \frac{(2t+1) \Delta_t \Delta_{k-s}}{\Delta_{k-s+t}} \,.\]
To see this, choose $T\in A_{s-t}^*$ uniformly, and condition on its configuration externally to $K$, its (unique) induced face that contains $k-s+t$ internal vertices.
Notice that $T\in A_s$ if and only if the triangulation of $K$ induces a face $K'$ with $k-s$ internal vertices, and that each of the $\Delta_{k-s+t}$ triangulations of $K'$ has an equal probability under the counting measure. Thus, the conditional probability that $T\in A_s$ is precisely $(2t+1)\Delta_t \Delta_{k-s} / \Delta_{k-s+t}$, where (as before) we triangulated $f$ with $t$ internal vertices (those not in $K'$), chose one of the $2t+1$ faces to contain $k-s$ vertices, and triangulated it accordingly.

Combining the last two identities shows that
\[ \frac{|A_s^*|}{\Delta_{k-s}} = \frac{|A_s| - \sum_{t=1}^{s-1} |A_{s-t}^*\cap A_s|}{\Delta_{k-s}} = (2s+1)\Delta_s  -\sum_{t=1}^{s-1} (2t+1)\Delta_t \frac{|A_{s-t}^*|}{\Delta_{k-s+t}} \,.\] 
We may thus define
\[ a_s := \frac{|A_s^*|}{\Delta_{k-s} \gamma^s}\qquad\mbox{($s=1,\ldots,\nu$)}\]
(so that $|A_{s-t}^*|/  \Delta_{k-s+t}$ in the last equation becomes $ a_{s-t} \gamma^{-s+t}$) and find that 
\[ a_s = (2s+1) \Delta_s \gamma^{-s} - \sum_{t=1}^{s-1} (2t+1)\Delta_t \gamma^{-t} a_{s-t}\,.\]
Further defining
\[ \phi_s := (2s+1) \Delta_s \gamma^{-s}\qquad\mbox{($s=1,2,\ldots$)}\]
allows us to rewrite the recursion on $a_s$ as
\[ a_s = \phi_s - \sum_{t=1}^{s-1}\phi_t a_{s-t}\,,\]
a relation through which the definition of $a_s$ extends to all $s\in\mathbb N$ (as $\phi_s$ is defined for all $s\in \mathbb N$).
Summing this over $s$ yields
\begin{equation}\label{eq:sum-as} \sum_{s=1}^\nu a_s = \sum_{s=1}^\nu \phi_s - \sum_{t=1}^{\nu-1}\phi_t \sum_{s=t+1}^\nu a_{s-t} = \sum_{s=1}^\nu  \bigg(1-\sum_{\ell=1}^{\nu-s}a_\ell\bigg)\phi_s\,.
\end{equation}
Since $\Delta_s \sim C_0 s^{-5/2} \gamma^s$ for a universal $C_0>0$ as $s\to\infty$, we find that $\phi_s \sim 2C_0 s^{-3/2}$ as $s\to\infty$, and hence $\sum_s \phi_s$ converges; moreover, we argue that \[ \phi:=\sum_{s=1}^\infty \phi_s =\frac{37}{27}\,.\]
Indeed, the asymptotics of $\Delta_s$ shows that the radius of convergence of $K(x) = \sum_{j=0}^\infty \Delta_j x^j$, the generating function for proper triangulations, is $1/\gamma$, and it converges uniformly in $[0,1/\gamma]$ (as $\Delta_j \gamma^{-k}= O(j^{-5/2})$). The same holds for $\sum_{j=1}^\infty j \Delta_j x^{j-1}$ (where $j \Delta_j = O(j^{-3/2})$ and there is still convergence at the boundary point $x=1/\gamma$). The latter corresponds to $K'(x)$ in $(0,1/\gamma)$, whence, by the definition of $\phi_s$, the fact that  $\sum_{s=1}^\infty \Delta_s\gamma^{-s} = K(\frac1\gamma) -1$, and the continuity of $K(x)+(2/\gamma)K'(x)$ as $x\to(1/\gamma)^-$, we find that
\[ \phi =  K(1/\gamma)-1 + (2/\gamma) K'(1/\gamma) \,.\]
Brown~\cite[Eqs.~(4.1)--(4.3)]{Brown64} showed that if $u(x)$ solves $x=u(1-u)^3$ then 
\[ K(x) = \frac{1-2u(x)}{(1-u(x))^{3}}\,,\]

We see that $\frac{d}{du}K(u)=(1-4u)(1-u)^{-4}$ and $\frac{d}{du}x(u)=(1-u)^2(1-4u)$, whence by the chain rule,
 $\frac{d}{dx}K(x) = (1-u)^{-6}$. Substituting $u=\frac14$, for which $x(u)=3^3/4^4 = 1/\gamma$, we thus conclude that
 \[ \phi = \frac{1-2u}{(1-u)^3}-1 + 2u(1-u)^3
 (1-u)^{-6}  = \frac1{(1-u)^3}-1 =  \frac{64}{27}-1=\frac{37}{27}\,.
\]
Since $0<a_s < \phi_s$ for every $s$, we also have $\sum_s a_s< \phi$ and write $a = \sum_{s=1}^\infty a_s$. 

Revisiting~\eqref{eq:sum-as}, we see that, as $\nu\to \infty$, its left-hand converges to $a$ whereas its right-hand converges to $(1-a)\phi$ (indeed, the right-hand is at least $(1-a)\sum_{s=1}^\nu \phi_s \to (1-a)\phi$ and at the same time it is at most $\sum_{s=1}^{\lfloor\nu/2\rfloor}(1-\sum_{\ell=1}^{\lceil\nu/2\rceil} a_\ell)\phi_s + \sum_{s=\lceil\nu/2\rceil}^\nu \phi_s \to (1-a)\phi$ by the convergence of $\sum_\ell a_\ell$ and $\sum_s\phi_s$).
 Rearranging, 
\[a=\frac\phi{\phi+1} = \frac{37}{64}\,.\]
Observe that 
\[ \Delta_{k-s} = (1+o(1))(1-s/k)^{-5/2}\Delta_k\gamma^{-s} = (1+O(\nu/k))\Delta_k \gamma^{-s}\,,\] where the error in the $O(\nu/k)$-term is uniform over $s$. From this we can infer that
\[ \frac1{\Delta_k}\sum_{s=1}^{\nu\wedge\sqrt k} \left|A^*_s\right| = (1+O(\nu/k))\sum_{s=1}^{\nu\wedge\sqrt{k}} a_s = (1+O(1/\sqrt{k})+o(1))a\,.
\]
At the same time, $\Delta_{k-s} \leq (1+o(1))2^{5/2}\Delta_k\gamma^{-s}$ using $s\leq\nu<k/2$; thus, for every $\epsilon>0$ there exists some~$L$ such that $\Delta_{k-s}<8\Delta_s\gamma^{-s}$ for all $s\geq L$, as well as $\sum_{s\geq L}a_s<\epsilon/8$, and as $\nu,\sqrt{k}\gg 1$ we get
\[ \frac{1}{\Delta_k} \sum_{s \geq \nu \wedge\sqrt k} |A_s^*| \leq 
\frac{1}{\Delta_k} \sum_{s \geq L} a_s \leq \epsilon\,, 
\]
as required.
\end{proof}

\subsubsection{Enumerating partial triangulations}

Denote the set of \emph{partial triangulations} (each viewed as a collected of triangles) by
\begin{equation}\label{eq:Sn-nu-def}\mathcal{P}^\nu_n:= \left\{P \subset \tbinom {[n]}3\,:\;\emptyset\neq P\subsetneq T \mbox{ for some } T\in\bigcup_k\cT^{\nu}_{k,n}\right\}\,.\end{equation}
Let $P\in\mathcal{P}^\nu_n$ be a partial triangulation with $f=|P|$ triangles. We consider a simplicial complex $X_P$ that contains the triangles of $P$, the vertices and edges that these triangles contain, and the outer cycle $(123)$. An edge in $X_P$ is called {\em internal} if it is contained in two triangles of $X_P$, and it is called a {\em boundary} edge otherwise. Similarly, A vertex in $X_P$ is called  {\em internal} if all the edges of $X_P$ that contain it are internal. Otherwise, either the vertex is one of the outer vertices $1,2,3$ or we call it a {\em boundary} vertex. We denote by $v_I,v_{\partial}$ the number of internal and boundary vertices respectively.  The {\em degree} of an edge in $X_P$ is the number of triangles of $X_P$ that contain it. We refer to the subgraph of $X_P$ that consists of the edges of degree smaller than $2$ and the vertices they contain as the {\em boundary graph} of $P$.

\begin{claim}\label{clm:vpartial-vi-exponent}
For every partial triangulation $P\in \mathcal P^\nu_n$ with $v_I$ internal vertices and $v_\partial$ boundary vertices, if 
    \[ \xi= |P|/2 - v_\partial -v_I\,,\]
    then $\xi \leq 0 \wedge (1/2-v_{\partial}/6)$, and moreover
     $\xi = 0$ only if $P$ contains at least $2\nu$ triangles.
\end{claim}
\begin{proof}
We consider the planar simplicial complex $X_P$ and denote by $e_j,~j=0,1,2,$ the number of edges of degree $j$ in $X_P$, and by $\beta_0,\beta_1$ the first two Betti numbers of $X_P$. Namely, $\beta_0$ is equal to the number of connected components of $X_P$ and $\beta_1$ is equal to the number of finite connected components of $X_P$'s planar complement. 

The fact that $3s=e_1+2e_2$ is clear, and we claim in addition that $3(\beta_1+1)\le 2e_0+e_1$. Indeed, we count incidences between the edges of $X_P$ and the connected components of its planar complement. On the one hand, every such component, including the infinite one, is incident with at least $3$ edges of $X_P$, and on the other hand, an edge that is contained in $j$ triangles of $X_P$ is incident with exactly $2-j$ components of its planar complement.
We compute the Euler characteristic of the planar simplicial complex $X_P$ in two ways:
\[
3+v_{\partial}+v_I-e_0-e_1-e_2+s = \beta_0-\beta_1\,,
\]
and deduce that
\[
\frac {|P|}2 -v_{\partial}-v_I =3-\beta_0+\beta_1-e_0-\frac{e_1}2\,.
\]
At this point, the inequalities $e_0+\frac {e_1}2 \ge\frac 32(\beta_1+1)$, $\beta_0\ge 1$ and $\beta_1 \ge 1$ allow us to derive that
\[
\frac {|P|}2-v_{\partial}-v_I 
\le \frac{1-\beta_1}2\le 0\,.\]
Furthermore, $|P|/2=v_{\partial}+v_I$ if and only if $X_P$ is connected and its planar complement has precisely one connected component whose boundary is triangular.
This triangular boundary forms a missing face in every triangulation $T$ that contains $P$, and if $T$ is $\nu$-canonical, it must have at least $\nu$ vertices in the exterior of this missing face. In other words, $P$ is a triangulation of the disk with at least $\nu$ internal vertices having one face removed. In particular, $P$  contains at least $2\nu$ triangles.

The boundary graph of $X_P$ has minimal degree $2$, therefore $e_0+e_1\ge v_{\partial}+3$. We again apply the inequalities $\beta_0\ge 1$ and $\beta_1\le \frac{2e_0+e_1}3-1$ to conclude that 
\[
\frac {|P|}2-v_{\partial}-v_I 
\le 1-\frac{2e_0+e_1}6 \le \frac 12-\frac{v_{\partial}}6\,,\]
completing the proof.
\end{proof}
\begin{corollary}
\label{cor:after_claim}
Let $k\ge m\ge 0$ be integers and $T$ a triangulation of $(123)$ with $k$ internal vertices. For every subset $U\subset V(T)$ of $m$ internal vertices there are at most $2(k-m)+1$ triangles of $T$ that are disjoint of $U$.
\end{corollary}
\begin{proof}
If $m=0$ then $U=\emptyset$ and all the $2k+1$ triangles of $T$ are disjoint of $U$. Otherwise, consider the partial triangulation $P$ that contains all the triangles of $T$ that are disjoint of $U$. Clearly, $v_I(P)+v_{\partial}(P)=k-m$, and by Claim \ref{clm:vpartial-vi-exponent}, $|P|\le 2(k-m)$.
\end{proof}

\begin{claim}\label{clm:vi-vpartial}
Let $v_I,v_{\partial}$ and $u$ be integers. The following holds:
\begin{enumerate}[(i)]
    
    \item Let $P\in\mathcal P^\nu_n$ be a partial triangulation with $v_{\partial}$ boundary vertices. The number of triangulations $T\in \mathcal Z_n^\nu$ where $P \subset T$ and $T\setminus P$ has $u$ vertices is at most $(\gamma(u+v_{\partial}))^{v_{\partial}}(\gamma n)^{u}$.
    
    \item The number of partial triangulations $P\in \mathcal P^\nu_n$ with $v_I$ internal vertices and $v_{\partial}$ boundary vertices is at most $(8(v_I+v_{\partial}))^{v_{\partial}}
    (\gamma n)^{v_I+v_{\partial}}$.
\end{enumerate}
\end{claim}
\begin{proof}
For Part~(i), denote by $V_{\partial}\subset[n]$ the set of labels of the boundary vertices of $P$. Let $T$ be a triangulation that contains $P$ with $u$ vertices outside $P$. Suppose we remove $P$'s internal vertices from $T$ and triangulate the interior of $P$ without internal vertices. This yields a labelled  triangulation of $(123)$ with $v_{\partial}+u$ vertices, that contains all the triangles of $T\setminus P$, in which $v_{\partial}$ vertices have labels in the set $V_{\partial}$ and the other $u$ vertices have arbitrary labels in $\{4,...,n\}$. Hence, the number of triangulations that contain $P$ is at most \[\Delta_{u+v_{\partial}}\binom{u+v_{\partial}}{u}v_{\partial}!n^u \le \left(\gamma(u+v_{\partial})\right)^{v_{\partial}}(\gamma n)^u. \]

For Part~(ii), let $P$ be a partial triangulation with $v_I$ internal vertices and $v_{\partial}$ boundary vertices. Consider a planar embedding of $X_P$ where $(123)$ is the boundary of the outer face and complete it to a triangulation without additional vertices. This yields a triangulation $T$ of $(123)$ with $v_I+v_{\partial}$ vertices that have labels in $\{4,...,n\}$ that contains all the triangles of $S$. 
    We claim that for every triangulation $T$ of $(123)$ with $v_I+v_{\partial}$ vertices, there are at most $\binom{v_I+v_{\partial}}{v_I}2^{2v_{\partial}+1}$ partial triangulations $P'\subset T$ with $v_I$ internal vertices and $v_{\partial}$ boundary vertices.
    Indeed, in order to construct $P'$ we first choose a subset $V_I\subset V(T)$ of $v_I$ vertices that will be the internal vertices, and then choose consistently the subset $R$ of $T$'s triangles that are not in~$P'$. Every triangle in $R$ must be disjoint of $V_I$ since $V_I$ are the internal vertices of $P'$, and by Corollary~\ref{cor:after_claim} there are at most $2v_{\partial}+1$ such triangles.
    In conclusion, the number of partial triangulations with $v_I$ internal vertices and $v_{\partial}$ boundary vertices is at most
    \[
    \Delta_{v_I+v_{\partial}}n^{v_I+v_{\partial}}\binom{v_I+v_{\partial}}{v_I}2^{2v_{\partial}+1}\le
    (8(v_I+v_{\partial}))^{v_{\partial}}
    (\gamma n)^{v_I+v_{\partial}}\,.\qedhere
    \]
\end{proof}

\subsection{First moment on triangulations below a given weight}

The following lemma is phrased for general linear subsets of the proper triangulations in order to support its application both to $\E[Z_n]$ and to $\E[Z_n^*]$ (see Corollary~\ref{cor:1st-moment}).

\begin{lemma}\label{lem:1st-moment-gen}
Let $\omega_n$ be as in~\eqref{eq:omega-window}, and set $\cK=\cK_{a_n}$ as in~\eqref{eq:K-def} for some $a_n\to\infty$ as $n\to\infty$.
Let $\mathscr{T}_k\subseteq\cT_k$ be such that $|\mathscr{T}_k|/\Delta_k \to \rho>0$ as $k\to\infty$, and let $\mathscr{Z}_n$, $\widetilde{\mathscr{Z}}_n$ be the respective analogs of $Z_n$, $\widetilde Z_n$ from~\eqref{eq:Zn-def}--\eqref{eq:tilde-Zn-def} w.r.t.\ $\mathscr{T}_k$ and its properly labeled counterpart $\mathscr{T}_{k,n}$. Then $\E \mathscr{Z}_n = \frac83 \sqrt{\frac2\pi}\rho e^A+o(1)$ and $\E [\mathscr{Z}_n-\widetilde{\mathscr{Z}}_n]=o(1)$.
\end{lemma}
\begin{proof}
For a given $T\in\mathscr{T}_{k,n}$, the law of $w(T)$ is that of a sum of $2k+1$ i.i.d.\ $\mathrm{Exp}(1)$ random variables, $w(T)$ is distributed as a $\mathrm{Gamma}(2k+1,1)$ random variable, thus for any $\omega_n>0$,
\[ \P(w(T) < \omega_n) = \P(\Po(\omega_n) \geq 2k+1)\,.\]
In particular, using $\P(\Po(\omega_n)=j)/\P(\Po(\omega_n)=j-1) = \omega_n/j$, 
for any $\omega_n = o(1)$ and $k\geq 1$ we have
\begin{equation}\label{eq:poisson-tail} \P(w(T) < \omega_n) = (1+O(\omega_n))e^{-\omega_n}\frac{\omega_n^{2k+1}}{(2k+1)!}
= (1+O(\omega_n))\frac{\omega_n^{2k+1}}{(2k+1)!}\,.
\end{equation}
Let $\mathscr{Z}_{k,n}=|\{T\in\mathscr{T}_{k,n}\,:\; w(T)\leq \omega_n\}$ (noting that with this notation $\mathscr{Z}_n = \sum_k \mathscr{Z}_{k,n}$ and $\widetilde{\mathscr{Z}}_{n}=\sum_{k\in\cK}\mathscr{Z}_{k,n}$).
Using that $|\mathscr{T}_k|=(\rho+o(1))\Delta_k$ for $\Delta_k=(\frac1{16}\sqrt{3/(2\pi)}+o(1))k^{-5/2}\gamma^{k+1}$ as in~\eqref{eq:Delta-k}, and accounting for the choice of $k$ labeled vertices out of $1,\ldots,n$, we have that, whenever $\omega_n=o(1)$,
\begin{equation}\label{eq:1st-moment-Z-k-mu}
\E \mathscr{Z}_{k,n} =\left(C_0 \rho+o(1)\right) \frac{n(n-1)\cdots(n-k+1)}{n^k}
\frac {1}{\sqrt{n} (4k)^{5/2}} \frac{(\omega_n\sqrt{\gamma n})^{2k+1}}{(2k+1)!}\qquad\mbox{for $C_0 = \sqrt{6\gamma/\pi}$} \,.\end{equation}
Therefore, if we write $\omega_n = \lambda_n/\sqrt{\gamma n}$ where $\lambda_n=\frac12\log n + \frac52\log\log n + A$ by the definition of $\omega_n$,
and further let $N_{\lambda_n}\sim\Po(\lambda_n)$, then we find that
\[ \E \mathscr{Z}_{k,n} \leq \left(C_0\rho+o(1)\right)\frac{e^{\lambda_n}}{\sqrt n (4k)^{5/2}} \P\left(N_{\lambda_n}=2k+1\right) = \left(C_0\rho e^A+o(1)\right)  \Big(\frac{\log n}{4k}\Big)^{\frac52}\P\left(N_{\lambda_n}=2k+1\right)\,.
 \]
Since $k\notin\cK$ implies that $|(2k+1)-\lambda_n|\geq (2-o(1))a_n\sqrt{\log n}$, 
\begin{align*} \E [ \mathscr{Z}_n-\widetilde{\mathscr{Z}}_n] &\leq \sum_{k\leq \frac18\log n}\E\mathscr{Z}_{k,n}+\sum_{\substack{k\geq \frac18\log n \\ k\notin \cK}}\E\mathscr{Z}_{k,n} \\ &\leq O\Big((\log n)^{\frac52}\P\left(N_{\lambda_n}\leq\tfrac14\log n+1\right)\Big)+ O\Big(\P\big(|N_{\lambda_n} - \lambda_n| > (2-o(1))a_n\sqrt{\log n }\big)\Big) = o(1)\,,
\end{align*}
 using the Poisson tail bound $\P(N_{\lambda_n} <\lambda_n-h) \vee \P(N_{\lambda_n}>\lambda_n + h) \leq \exp[-\frac{h^2}{h+\lambda_n}]$ 
 for the last transition to show the first $O(\cdot)$ term (taking $h=(\frac14+o(1))\log n$) is $n^{-1/8+o(1)}$ whereas the second one is $o(1)$ as $a_n\to\infty$.
 It thus suffices to show that $\E\widetilde{\mathscr{Z}}_n= \frac12 C_0\rho e^A+o(1)$. With $K_1 := \lceil \frac14\log n - a_n\sqrt{\log n} \rceil$, our upper bound on $\E \mathscr{Z}_{k,n}$ implies that
\begin{align*} \E \widetilde{\mathscr{Z}}_n \leq \left(C_0\rho e^A+o(1)\right)  \sum_{k \geq K_1} \Big(\frac{\log n}{4k}\Big)^{\frac52}\P\left(N_{\lambda_n}=2k+1\right) \leq \frac12 C_0\rho e^A +o(1)\,,
 \end{align*}
since $\P(N_{\lambda_n}\mbox{ odd}) = \frac12(1-e^{-2\lambda_n})$.
Conversely, we infer from~\eqref{eq:1st-moment-Z-k-mu} that for any $k=o(\sqrt{n})$,
\[ \E \mathscr{Z}_{k,n} =(C_0\rho+o(1)) \frac {1}{\sqrt n (4k)^{5/2}}
\frac{\lambda_n^{2k+1}}{(2k+1)!} = (C_0\rho e^A+o(1)) \Big(\frac{\log n}{4k}\Big)^{\frac52}\P(N_{\lambda_n}=2k+1)\,;
\]
hence, for $K_1$ as above and a corresponding definition of $K_2 = \lceil\frac14\log n + a_n\sqrt{\log n}\rceil$, 
\[ \E \widetilde{\mathscr{Z}_{n}} = \sum_{k=K_1}^{K_2} \E \mathscr{Z}_{k,n} \geq (C_0\rho e^A +o(1)) \P\left(N_{\lambda_n} \mbox{ odd }\,,\, 2K_1 +1\leq N_{\lambda_n} \leq 2K_2+1 \right)\,,\]
which is $\frac12C_0\rho e^A+o(1)$
by the same two estimates for a $\Po(\lambda_n)$ random variable used in the upper bound. 
\end{proof}

\begin{remark}
\label{rem:Ka-with-fixed-a}
The proof of Lemma~\ref{lem:1st-moment-gen} in fact shows that when considering $\cK_a$ for any fixed $a>0$ we have that $\E[\mathscr{Z}_n-\widetilde{\mathscr{Z}}_n] = O(e^{-(8-o(1))a^2})$, and in particular, when taking $\mathscr{T}_k=\cT_k$ (so that $\mathscr{Z}_n=Z_n$) we find that
\[ \lim_{a\to\infty} \limsup_{n\to\infty}\P\Big(\min_{T\in\bigcup_{k\notin\cK_a}\cT_{k,n}}w(T) \leq \omega_n\Big) = 0\,.\]
\end{remark}
Applying Lemma~\ref{lem:1st-moment-gen} to $\mathscr{T}_k=\cT_k$ ($\rho=1$) and $\mathscr{T}_k=\cT_k^\nu$ ($\rho = (\frac34)^3$ via Lemma~\ref{lem:number-canonical})
yields the following.
\begin{corollary}\label{cor:1st-moment}
In the setting of Theorem~\ref{thm:poisson} we have
$ \E Z_n =\frac83 \sqrt{2/\pi}e^A + o(1)$ and $ \E Z^*_n = \frac9{4\sqrt{2\pi}} e^A+o(1)$,
while $Z_n - \widetilde{Z}_n\to 0$ and $Z^*_n - \widetilde{Z}^*_n\to 0 $ in probability as $n\to\infty$.
\end{corollary}

\subsection{Second moment for canonical triangulations}
The following lemma, proved via a second moment argument, is the main ingredient in establishing the Poisson weak limit of $Z_n^*$ in Theorem~\ref{thm:poisson}.
\begin{lemma}\label{lem:2nd-moment}
Let $\omega_n$ as in~\eqref{eq:omega-window}, let $\nu = \lfloor \frac1{10}\log n\rfloor$, and let $a_n$ be so that $a_n=o(\sqrt{\log n})$ and $\lim_{n\to\infty}a_n=\infty$. Then $\widetilde{Z}_n^*$ as defined in~\eqref{eq:tilde-Zn-def} w.r.t.\ $\omega_n$ and $\cK_{a_n}$ satisfies $ \var(\widetilde{Z}_n^*) \leq \E \widetilde{Z}_n^* + o(1)$.
\end{lemma}
\begin{proof}
For brevity, let $\widetilde\cT$ denote $ \bigcup_{k\in\cK_a} \cT^\nu_{k,n}$. 
We have
\[\E[(\widetilde{Z}_n^*)^2] = \sum_{T_1,T_2\in\widetilde\cT}\!\!\! \P(w(T_1)\le \omega_n,\, w(T_2)\le \omega_n) \leq \E[\widetilde{Z}_n^*] + (\E \widetilde{Z}_n^*)^2 + \sum_{\substack{T_1 \neq T_2\in\widetilde\cT \\ T_1\cap T_2\neq\emptyset}} 
\!\!\!\P(w(T_1)\le \omega_n,\, w(T_2)\le \omega_n)\]
(as the summation over $T_1=T_2$ contributes $\E \widetilde{Z}_n^*$ and the summation over disjoint pairs $T_1\cap T_2=\emptyset$ is bounded by $(\E \widetilde{Z}_n^*)^2$ since the events $w(T_1)\leq \omega_n$ and $w(T_2)\leq \omega_n$ for such a pair $(T_1,T_2)$ are independent); thus, recalling the definition of $\mathcal{P}_n^\nu$ from~\eqref{eq:Sn-nu-def}, 
\begin{equation}
    \label{eq:2nd-moment}
\var(\widetilde{Z}_n^*) \leq \E \widetilde{Z}_n^* + \Xi\qquad\mbox{where}\qquad\Xi := \sum_{P\in\mathcal P^\nu_n}\sum_{\substack{T_1 \neq T_2\in \widetilde\cT,\\ T_1\cap T_2=P}}\P(w(T_1)\le \omega_n,\,w(T_2)\le \omega_n)\,,
\end{equation}
and we aim to prove $\Xi = o(1)$. 
Consider two triangulations $T_1 \neq T_2$ in $\widetilde\cT$  with $k_1$ and $k_2$ internal vertices, respectively, where $P=T_1\cap T_2$ contains $s>0$ triangles (notice that $s\leq 2(k_1\wedge k_2)+1$). In this case, the three variables $w(P)\sim\mathrm{Gamma}(s,1)$ and $w(T_i)-w(P)\sim\mathrm{Gamma}(2k_i+1-s,1)$ for $i=1,2$ are independent. Moreover, $\P(w(T_i)-w(P)< x) $ is equal to $  (1+O(x))x^{2k_i+1-s}/(2k_i+1-s)!$ just as we used in~\eqref{eq:poisson-tail}, so
\begin{align*}
    \P(w(T_1)\le \omega_n,\,w(T_2)\le \omega_n) 
    &=(1+o(1))
\int_{0}^{\omega_n}\frac{x^{s-1}}{(s-1)!}\cdot \frac{(\omega_n-x)^{2k_1+1-s}}{(2k_1+1-s)!}\cdot\frac{(\omega_n-x)^{2k_2+1-s}}{(2k_2+1-s)!} \, dx\\
&= (1+o(1))\frac{{\omega_n}^{2k_1+2k_2+2-s}(2k_1+2k_2+2-2s)!}{(2k_1+2k_2+2-s)!(2k_1+1-s)!(2k_2+1-s)!}\,.
\end{align*}
To bound the factorial terms above, note that
\begin{align*}
\frac{(2k_1+2k_2+2-2s)!(2k_1+1)!(2k_2+1)!}{(2k_1+2k_2+2-s)!(2k_1+1-s)!(2k_2+1-s)!}
&=
\prod_{j=0}^{s-1}\frac{(2k_1+1-j)(2k_2+1-j)}{2k_1+2k_2+2-s-j}
\leq \prod_{j=0}^{s-1}(2k_1+1-j)\\
&\leq (2k_1+1)^s \exp\left[-\tbinom{s}2/(2k_1+1)\right]\,,
\end{align*}
where the first inequality used that $s \leq 2k_1+1$ (and hence $2k_2+1-j\leq 2k_1+2k_2+2-s-j$). 

Let $\lambda_n = \omega_n \sqrt{\gamma n}$ (so that $\lambda_n = \frac12\log n + \frac52\log\log n + A$).
Since $T_1\in\widetilde\cT$ (and so $k\in\cK_{a_n}$), we have  $2k_1+1=(\frac12+o(1))\lambda_n $, and moreover $2k_1+1\leq \frac12\log n + 2a_n\sqrt{\log n}+1 \leq (1+(4+o(1))a_n/\sqrt{\log n})\lambda_n$, so 
\begin{align*}
    \P(w(T_1)\le \omega_n,\,w(T_2)\le \omega_n) \leq \frac{\omega_n^{2k_1+2k_2+2-s}}{(2k_1+1)!(2k_2+1)!} \lambda_n^s \exp\bigg((4+o(1)){\frac{a_n}{\sqrt{\log n}}}s-(1-o(1))\frac{s^2}{\log n}\bigg)\,.
\end{align*}
Revisiting~\eqref{eq:2nd-moment}, write 
$ \Xi = \sum_{k_1,k_2\in \cK_{a_n}} \sum_\rho \Xi_\rho$
where $\rho$ goes over tuples of integers $(v_I,v_\partial,u_1,u_2,s)$ such that such that $k_1=v_I+v_{\partial}+u_1$ and $k_2=v_I+v_{\partial}+u_2$, and $\Xi_{\rho}$ is the contribution to $\Xi$ by pairs $(T_1,T_2)$ where $T_1\neq T_2$ and $P=T_1\cap T_2$ has $v_I$ internal vertices, $v_\partial$ boundary vertices and $s$ triangles. For a given~$\rho$, we bound the number of choices for $P\in\mathcal{P}_n^\nu$ and $T_1,T_2\in\widetilde\cT$ consistent with it via Claim~\ref{clm:vi-vpartial}, whence the above bound on $\P(w(T_1)\leq\omega_n),\,w(T_2)\leq\omega_n)$ yields
\begin{align*} \Xi_\rho &\leq (9\log n)^{3v_\partial} (\gamma n)^{v_\partial+v_I + u_1+u_2} \frac{\omega_n^{2k_1+2k_2+2-s}}{(2k_1+1)!(2k_2+1)!} \lambda_n^s \exp\bigg((4+o(1))\frac{a_n}{\sqrt{\log n}}s-(1-o(1))\frac{s^2}{\log n}\bigg)\,.
\end{align*}
Observing that 
\[
v_{\partial} + v_I+u_1+u_2 - (k_1+k_2+1-s/2) = s/2-1-v_{\partial}-v_I\,,
\]
we see that
\begin{align*} (\gamma n)^{v_\partial+v_I + u_1+u_2} \frac{\omega_n^{2k_1+2k_2+2-s}}{(2k_1+1)!(2k_2+1)!} &=(\gamma n)^{s/2-1-v_{\partial}-v_I}
\frac{{\lambda_n}^{2k_1+2k_2+2-s}}{(2k_1+1)!(2k_2+1)!}
\\
&\leq (\gamma n)^{s/2-1-v_{\partial}-v_I} \lambda_n^{-s}\exp(2\lambda_n)\,,
\end{align*}
where the inequality between the lines used $\lambda_n^m \le e^{\lambda_n} m!$ for $m=2k_1+1$ and $m=2k_2+1$. Plugging this, as well as the fact that $\exp(\lambda_n) = e^A \sqrt{n} (\log n)^{5/2}$, in our bound for $\Xi_\rho$, we deduce that
\begin{align}\Xi_\rho\leq e^{2A} (\gamma n)^{s/2-v_{\partial}-v_I} (9\log n)^{3v_{\partial}} (\log n)^{5} \exp\left( (4+o(1))\frac{a_n}{\sqrt{\log n}}s -\frac {s^2}{\log n}\right)\,. \label{eq:2nd-moment-last}
\end{align}
Note that the exponent in~\eqref{eq:2nd-moment-last} is  maximized at $s=(2+o(1))a_n \sqrt{\log n }$, and so
 \[\exp\left( (4+o(1))\frac{a_n}{\sqrt{\log n}} s -\frac {s^2}{\log n}\right) \leq  e^{(4-o(1))a_n^2} = n^{o(1)}\]
 by our assumption that $a_n = o(\sqrt{\log n})$.
 
To bound right-hand of~\eqref{eq:2nd-moment-last}, we now appeal to Claim~\ref{clm:vpartial-vi-exponent}, and consider the following cases.
\begin{enumerate}
    \item If $v_{\partial} \geq 4$ then
    \[
    \Xi_\rho\leq \sqrt{\gamma n}\left(\frac{(9\log n)^{3}}{(\gamma n)^{1/6}}\right)^{v_\partial} n^{o(1)}\leq n^{-1/6+o(1)}\,.
    \]
    \item If $v_{\partial} \le 3$ and $s/2-v_I-v_{\partial}\le -1/2$ then
    \[\Xi_\rho \leq \frac{(9\log n)^{3v_{\partial}}}{ \sqrt{\gamma n}} n^{o(1)} \leq n^{-1/2+o(1)}\,.\]
    \item If $v_{\partial} \leq 3$ and $s/2-v_I-v_{\partial}= 0$ then  $s\geq 2\nu = (\frac15+o(1))\log n$, hence~\eqref{eq:2nd-moment-last} implies that 
    \[
    \Xi_\rho \leq e^{2A}(9\log n)^{3v_{\partial}} (\log n)^5 \exp\left(-(\tfrac1{25}-o(1)) \log n\right) = n^{-1/25+o(1)}\,.
    \]
\end{enumerate}

The proof is concluded by observing that there are $O(a_n^2 \log n)$ choices for the pair $(k_1,k_2)$, at which point there are $O(\log^2 n)$ choices for $\rho$ (determined by $v_I,v_\partial$) in the summation $\Xi=\sum_{k_1,k_2\in\cK_{a_n}}\sum_\rho \Xi_\rho$.
\end{proof}

\subsection{Proofs of Theorem~\ref{thm:poisson} and the main results on proper triangulations}
We are now ready to derive Theorem~\ref{thm:poisson}, and read from it Theorem~\ref{mainthm:gumbel} and our main results in Theorem~\ref{mainthm:tight} concerning proper triangulations.
\begin{proof}[\textbf{\emph{Proof of Theorem~\ref{thm:poisson}}}]
In light of Corollary~\ref{cor:1st-moment}, it suffices to show that $\widetilde Z_n^*\stackrel{\mathrm{d}}\to \Po(\lambda)$ for $\lambda = \frac{9}{4\sqrt{2\pi}}e^A$. 

Moreover, since Corollary~\ref{cor:1st-moment} established that $\lim_{n\to\infty}\E[Z_n^*]= \lambda$ independently on the sequence $\nu$ as long as $\nu\to\infty$ with $n$ and $\nu\leq \overline\nu:=\lfloor \frac1{10}\log n\rfloor$, we see that the number of triangulations $T$ with $w(T)\leq\omega_n$ that are $\nu$-canonical but not $\overline\nu$-canonical is converging to $0$ in probability. We thus assume w.l.o.g.\ that $\nu=\overline\nu$.

Writing $I_T = \one\{w(T)\leq\omega_n\}$, so that $\widetilde{Z}_n^* = \sum_{k\in\cK_{a_n}} \sum_{T\in\cT^\nu_{n,k}} I_T$, observe the number of internal vertices~$k$ in $T$ satisfies $k=(\frac14+o(1))\log n$ by the definition of $\cK_{a_n}$ and the fact $a_n=o(\sqrt{\log n})$, whence by~\eqref{eq:poisson-tail},
\[ \P(I_T)  = (1+o(1))\frac{\omega_n^{2k+1}}{(2k+1)!} 
=n^{-(\frac14 + o(1))\log n} = o(1)
\,.\]
By the Stein--Chen method, applied to the sum of indicators $I_T$ that are positively related in the product space $\{w_x\,:\;x\in\binom{[n]}3\}$
(see~\cite[Thm.~5]{Janson94} and~\cite{BHJ92} for more details, as well as~\cite[Prop.~2]{Janson94} stating that, by the FKG inequality, increasing functions of independent random variables are positively related)% also~\cite[Theorem~6.24]{JLR00} and~\cite[Theorem~6.27]{JLR00})
, we find that
\[ \left\| \P(\widetilde Z_n^*\in\cdot)-\P( \Po(\E \widetilde Z_n^*)\in\cdot)\right\|_{\textsc{tv}} \leq \frac{\var(\widetilde Z_n^*)}{\E \widetilde Z_n^*} - 1 + 2\max_T \P(I_T) = o(1)\,,
\]
where the last equality used Lemma~\ref{lem:2nd-moment}.
\end{proof}

\begin{proof}[\textbf{\emph{Proofs of Theorems~\ref{mainthm:tight} and~\ref{mainthm:gumbel} for proper triangulations}}]
To prove Theorem~\ref{mainthm:gumbel}, let $W_n^*$ be the minimum of $\{w(T): T\in\bigcup_k\cT_{k,n}^{(\nu)}\}$ and set $Y_n^*$ as per~\eqref{eq:Wn*-decomopostion}. For $\omega_n$ as in~\eqref{eq:omega-window} with $A=\log(\frac49\sqrt{2\pi})-y$, we get
\[ \P(Y_n^* \leq y) = \P\bigg(W_n^* \geq \frac{\frac12 \log n + \frac52 \log\log n + \log(\tfrac49\sqrt{2\pi})-y}{\sqrt{\gamma n}}\bigg)
= \P(W_n^* \geq\omega_n)\to e^{-e^{-y}}
\]
as $n\to\infty$, using that $\P(W_n^* \geq \omega_n) = \P(Z_n^* = 0)$ and that $Z_n^*\stackrel{\mathrm d}\to \Po(\lambda)$ for $\lambda=\frac9{4\sqrt{2\pi}}e^A = e^{-y}$.

For Theorem~\ref{mainthm:tight}, let $D_n$ be the minimum weight over all proper triangulations of $(123)$. Theorem~\ref{mainthm:nonplanar} will establish that $W_n = D_n$ with high probability, hence it suffices to consider $D_n$.
Tightness is then readily derived by observing that, on one hand, for $A\in\R$
\[ \P(D_n\leq \omega_n) = \P(Z_n>0) = O(e^A) \to 0\quad\mbox{as $A\to-\infty$}\,,\]
whereas on the other hand, by the above result on $W_n^*$ we have that
\[ \P(D_n \geq \omega_n) \leq \P(W_n^*\geq\omega_n) = \exp\left(-O(e^{A})\right) \to 0\quad\mbox{as $A\to\infty$}\,.
\]
Finally, for the statement on the number of vertices in the minimizing triangulation, let $\epsilon>0$ and fix $A>0$ large enough so that $\P( D_n \leq (\frac12\log n + \frac52\log\log n + A)/\sqrt{\gamma n})\geq 1-\epsilon$, using the tightness established above.
Set $\omega_n$ as in~\eqref{eq:omega-window} with this $A$, and take $a>0$ large enough so that $\P(Z_n = \widetilde{Z}_n) > 1-\epsilon$ via Remark~\ref{rem:Ka-with-fixed-a}. Combining these, the a.s.\ unique minimizer $T$ belongs to $\bigcup_{k\in\cK_a}\cT_{k,n}$ with probability at least $1-2\epsilon$.
\end{proof}

\begin{figure}
    \centering
     \begin{tikzpicture}[vertex/.style = {inner sep=2pt,circle,draw,fill,label={#1}}]
 	\newcommand\rad{2}	

	\begin{scope}
  \coordinate (v1) at (90:\rad);
  \coordinate (v2) at (210:\rad);
  \coordinate (v3) at (330:\rad);
  \coordinate (x) at (0,0);
   
    \draw [thick,black] (v1.center)--(v2.center)--(v3.center)--cycle;

	\path [fill=gray!30] (v1.center)--(v2.center)--(x.center)--cycle;
    
	\foreach \y in {v1, v2, v3}
		\draw [black] (x.center)--(\y.center);
    
  	\node[vertex={above:$1$}] at (v1) {};
  	\node[vertex={below:$2$}] at (v2) {};
  	\node[vertex={below:$3$}] at (v3) {};
	\node[vertex={left:$x$}] at (x) {};
	\end{scope}
	
	\begin{scope}[shift={(3*\rad,0)}]
   \coordinate (v1) at (90:\rad);
   \coordinate (v2) at (210:\rad);
   \coordinate (v3) at (330:\rad);
   \coordinate (v4) at ($0.35*(v1) + 0.45*(v2) + 0.2*(v3)$);
   \coordinate (v5) at ($0.35*(v1) + 0.2*(v2) + 0.45*(v3)$);
   
     \draw [thick,black] (v1.center)--(v2.center)--(v3.center)--cycle;

 	\path [fill=gray!30] (v1.center)--(v4.center)--(v5.center)--cycle;
    
 	\foreach \x/\y/\z in {v1/v4/v5, v2/v4/v5, v2/v5/v3}
 		\draw [black] (\x.center)--(\y.center)--(\z.center)--cycle;
    
   	\node[vertex={above:$1$}] at (v1) {};
   	\node[vertex={below:$2$}] at (v2) {};
   	\node[vertex={below:$3$}] at (v3) {};
 	\node[vertex={left:$x$}] at (v4) {};
 	\node[vertex={right:$y$}] at (v5) {};
 	\end{scope}
    \end{tikzpicture}
     \caption{Different limiting laws arising from non-canonical triangulations: we optimize the total weights of the white triangles, and then triangulate the gray triangles via minimal proper canonical triangulations. Left: $Y_n'\stackrel{\mathrm d}\to\log(4\sqrt{2\pi}/9)-\mathrm{Gumbel} +\sqrt\gamma\,\mathrm{Rayleigh(1)}$;  Right: $Y_n''\stackrel{\mathrm d}\to\log(4\sqrt{2\pi}/9)-\mathrm{Gumbel} +\sqrt\gamma\,\mathrm{Weibull(4,(4!)^{1/4})}.$  
     }
    \label{fig:non-canon}
\end{figure}
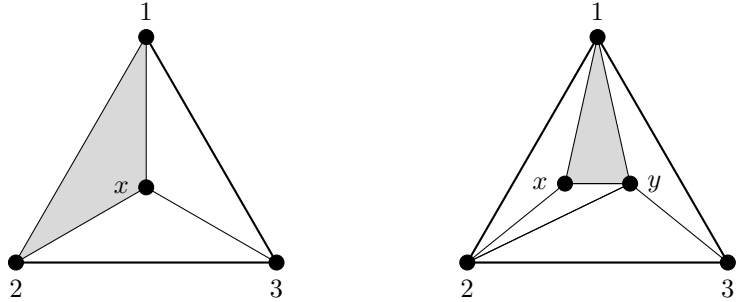

\begin{remark}\label{rem:other-dist}
We saw that the variable $Y_n$ from Theorem~\ref{mainthm:tight} has $Y_n \stackrel{\mathrm d}\to \mathfrak{G}$, where $ \log(\frac49\sqrt{2\pi})-\mathfrak G$ is Gumbel, when restricting the triangulations to canonical ones (which represent a constant fraction of the triangulations with~$k$ internal vertices for every $k$). Non-canonical triangulations may lead to other limits for $Y_n$, e.g.:
\begin{compactitem}[\indent$\bullet$]
    \item Consider the proper triangulations obtained by first triangulating $(123)$ via a single internal vertex $x$
    minimizing $Z_x:=w(13x)+w(23x)$, then canonically triangulating only $(12x)$ with any number of internal vertices. The choice of $x$ is independent of the weights of the inner triangles in $(12x)$ (whose minimum weight is distributed as $W_{n-2}^*$), and furthermore $\{Z_x : x=4,\ldots,n\}$ are i.i.d.\ $\mathrm{Gamma}(2,1)$. Hence, if we write the minimum weight $W'_n$ of such triangulations as in~\eqref{eq:Wn-decomposition}, then the corresponding variable~$Y'_n$ satisfies $ Y'_n \stackrel{\mathrm d}\to \mathfrak{G} +\sqrt\gamma\,\mathfrak{R} $ 
where $\mathfrak{G}$ is the shifted Gumbel from above and $\mathfrak{R}\sim\mathrm{Rayleigh(1)}$. 
    \item Consider the proper triangulations obtained by first triangulating $(123)$ via two internal vertices $x,y$
    minimizing $Z_{xy}:=w(12x)+w(23y)+w(13y)+w(2xy)$ (see Fig.~\ref{fig:non-canon}), then canonically triangulating only $(1xy)$ with any number of internal vertices. Again, the choice of $x,y$ is independent of the minimum weight triangulation of $(1xy)$ (distributed as $W_{n-3}^*$), and $\{Z_{xy} : 4\leq x\neq y \leq n\}$ are i.i.d.\ $\mathrm{Gamma}(4,1)$. Thus, writing the minimum weight $W''_n$ of such triangulations as in~\eqref{eq:Wn-decomposition}, the corresponding variable~$Y''_n$ has~$ Y''_n \stackrel{\mathrm d}\to \mathfrak{G} +\sqrt\gamma\,\mathfrak{W} $ where $\mathfrak{G}$ is the shifted Gumbel from above and $\mathfrak{W}\sim\mathrm{Weibull}(4,(4!)^{\frac14})$. 
\end{compactitem}
\end{remark}

\section{Homological and homotopical fillings}\label{sec:fillings}
In this section we prove Theorem~\ref{mainthm:nonplanar} which completes the proof of Theorem~\ref{mainthm:tight}. We start by describing a known characterization of $\F_2$-homological fillings as surface fillings --- the faces of triangulated surfaces of an arbitrary genus. Afterwards, we compute the expected number of surface fillings below a given weight.

\subsection{Surface fillings}

A surface is a connected $2$-dimensional manifold. A triangulation of a surface $S$ is a connected graph $G$ that can be embedded on $S$ such that all the components of $S\setminus G$ are homeomorphic to a disk and their boundary is a triangle of $G$. These triangles are called the {\em faces} of the triangulation. If such a triangulated surface has $v$ vertices, $e$ edges, and $f$ faces then its Euler Characteristic is $\chi=v-e+f$ and its {\em type} is $g:=1-\chi/2$. A triangulated surface is called {\em orientable} if it is possible to orient the faces of the surface, i.e., to choose a direction around the boundary of every face, such that every edge is given opposite directions by the two faces that contain it. In such case, the type of the surface is a non-negative integer that is equal to its genus. For example, the sphere and the torus are orientable surfaces of types $0$ and $1$ respectively. On the other hand, if $S$ is non-orientable then its type is a positive integer or half-integer. For instance, the projective plane and the Klein bottle are non-orientable surfaces of types $1/2$ and $1$ respectively.

We always consider {\em rooted} triangulated surfaces where the vertices of one of the faces are labeled $1,2$ and~$3$. In order to be consistent with the terminology in the rest of the paper, we refer to the remaining vertices and faces as {\em internal}. Denote by $\mathcal S_{g,f}$ the set of rooted triangulated surfaces of type $g$ with $f$ internal faces, where the internal vertices are unlabeled. For instance, $\mathcal S_{0,2k+1}$ is isomorphic to $\cT_k$ by the equivalence of planar and spherical triangulations. We note that if $g$ is integral then $\mathcal S_{g,f}$ contains both orientable and non-orientable triangulated surfaces of type $g$.
Any triangulated surface $S\in\mathcal S_{g,f}$ has
\begin{equation}
    \label{eqn:gfk}
    k=\frac f2-\frac12 -2g
\end{equation}
internal vertices. This number is determined by Euler's formula $(k+3)-e+(f+1)=2-2g$ and by the standard double counting $2e=3(f+1)$.

A {\em face-proper labeling} of a rooted triangulated surface is a labeling of the vertices in which (i) no two adjacent vertices have the same label, and (ii) no two faces have the same $3$ labels. In other words, a labeling is face-proper if it induces an injective function from the faces of the surface to $\binom {[n]}3$. 
A {\em surface filling} of $(123)$ over $\{1,...,n\}$ is a face-proper labeled rooted triangulated surface. Denote by $\mathcal{S}_{g,f,n}$ the set of surface fillings of $(123)$ whose underlying triangulated surface belongs to $\mathcal S_{g,f}$. We view a surface filling of $(123)$ as a subset of $\binom{[n]}3$ that consists of the labels of the internal faces of $S$. Accordingly, given weights $\{w_x\,:\;x\in\binom {[n]}3\}$ we denote $w(S)=\sum_{x\in S}w_x.$ Note that a surface filling can be improper in the sense that two vertices in the surface may have the same label, as long as the labels of the faces are distinct. For instance, even though $\mathcal S_{0,2k+1}$ and $\cT_k$ are isomorphic, $\mathcal S_{0,2k+1,n}$ strictly contains $\cT_{k,n}$ if $k\ge 3$ since for some surfaces $S\in\mathcal S_{0,2k+1}$ there is a face-proper labeling that is improper.

We end this discussion with the equivalence of surface fillings and
$\F_2$-homological fillings (recall that an $\F_2$-homological filling of $(123)$ is a vector whose support is a set of faces in which every edge in $\binom {[n]}2$ appears an even number of times except the boundary edges $\{12,23,13\}$). 
The characteristic vector of every surface filling $S$ of $(123)$ is clearly an $\F_2$-homological filling of $(123)$. Indeed, every edge in the triangulated surface $\tilde S$ underlying $S$ is contained in two of its faces and  $S$ consists of the labels of these faces except one face that is labeled $123$. Therefore, every edge in $\binom {[n]}2$ appears an even number of times in $S$ except the boundary edges. The following special case of Steenrod's Problem \cite{eilenberg49} asserts the converse statement. (In what follows, call an $\F_2$-homological filling $z$ \emph{inclusion-minimal} if there is no such  filling $z'\neq z$ such that $z\leq z'$ point-wise.)
\begin{fact}\label{fact:steenrod}
Every inclusion-minimal $\F_2$-homological filling of $(123)$ is the characteristic vector of a surface filling of $(123)$.
\end{fact}
\begin{proof}
Let $z$ be an inclusion-minimal $\F_2$-homological filling of $(123)$ and $Z=\mathrm{supp}(z)\cup\{123\}\subset \binom{[n]}{3}$. Consider a disjoint union $X=\bigcup_{ijk\in Z}X_{ijk}$ of $|Z|$ triangles where the vertices of the triangle $X_{ijk}$ are labeled by $i,j,k$. Every pair $ij\in\binom{[n]}2$ appears as an edge in an even number of triangles in $X$. We choose, for every such pair $ij$, an arbitrary matching of the edges labeled $ij$ in $X$ and glue them accordingly (for each matched pair of edges, identifying the two vertices labeled $i$ with one another and the two vertices labeled $j$ with one another). Note that the gluing did not affect the labels of the faces, and therefore the obtained space $S$ is a face-proper labeled triangulated surface where one of the faces is labeled $(123)$. Indeed, every gluing of a disjoint union of triangles according to any matching of the edges yields a disjoint union of surfaces. In addition, $S$ is connected since $z$ is inclusion-minimal.
\end{proof}

\subsection{First moment on non-spherical surface fillings}
A surface filling $S$ of $(123)$ is called {\em non-spherical} if the surface underlying $S$ is of type $g>0$ or the labeling of $S$ is improper.
Theorem~\ref{mainthm:nonplanar} will follow immediately from Fact~\ref{fact:steenrod} once we prove the following lemma, since every surface filling whose underlying surface is a sphere and whose labeling is proper is a proper triangulation of $(123)$. 
\begin{lemma}
\label{lemma:surfaces}
Assign i.i.d.\ $\mathrm{Exp}(1)$ weights to the $\binom n3$ triangles on $n$ vertices. Then, with high probability, the weight of every non-spherical surface filling of $(123)$ is at least $\frac32\log n/\sqrt{\gamma n}$.
\end{lemma}

Similarly to Lemma~\ref{lem:1st-moment-gen}, we prove Lemma~\ref{lemma:surfaces} by a first-moment argument. We replace Tutte's enumeration of planar triangulations with two  theorems concerning enumeration of triangulations and triangular maps on surfaces of high genus: Theorem~\ref{thm:gao} will give an asymptotically sharp estimation of $|\mathcal S_{g,f}|$ for small $g$, and from Theorem~\ref{thm:bud} we will derive an upper bound on $|\mathcal S_{g,f}|$ for large $g$.
\begin{theorem}[{\cite[Theorem~7]{Gao93}}]\label{thm:gao}
There exist constants $C_1,C_2>0$ such that for every fixed integer or half-integer $g\ge 0$,
\[
|\mathcal S_{g,f}| \le (C_1+o(1))(C_2 k)^{5(g-1)/2}\gamma^{k}\,,
\]
where $k=f/2-1/2-2g$, and the $o(1)$-term vanishes as $k\to\infty$.
\end{theorem}
The second enumeration result we use requires a slight modification for our purposes. The theorem states an asymptotically tight enumeration of triangular maps on orientable surfaces of any genus. A triangular map is the generalization of a triangulation where the the embedded graph is allowed to have self loops and double edges. 
\begin{theorem}[{\cite[Theorem~3]{Bud19}}]\label{thm:bud}
Let $f\in\Z$ let $g=g(f)\in\Z$ be  such that $g/f\to\theta\in[0,1/4]$ as $f\to\infty$. Then, the number of rooted triangular maps with $f$ faces on the orientable surface of type $g$ is equal to $
\left( f/2\right)^{2g}{(\tilde C+o(1))}^{f},
$ as $f\to\infty$,
where $\tilde C=\tilde C(\theta)$ is uniformly bounded over $\theta$.
\end{theorem}
Since a triangulation is a special case of a triangular map, we can use the bound in Theorem~\ref{thm:bud} as an upper bound for the number of orientable triangualted surfaces in $\mathcal S_{g,f}$. In order to bound the number of non-orientable surfaces in $\mathcal S_{g,f}$, we use the following observation. Every non-orientable triangulated surface $S$ of type $g$ with $f$ faces admits a double-covering by an orientable triangulated surface $\hat S$ of type $2g-1$ with $2f+2$ faces that is called the {\em orientation covering} (cf.~\cite[\S3.3]{hatcher}). Moreover, every such orientable triangulated surface $\hat S$ double-covers at most $6(2f+1)$ triangulated surfaces $S$. Indeed, such an $S$ is a obtained from $\hat S$ by matching its outer face to one of the remaining $2f+1$ faces, and identifying the three vertices of these two faces via one of the $3!$ permutations. Henceforth, the entire 2-covering will be uniquely determined (as every edge is contained in precisely two faces and $\hat S$ is connected).
This implies the following rough bound:
\begin{corollary}\label{cor:bud}
There exist constants $C,f_0>0$ such that for every nonnegative $g\in \frac12\Z$ and integer $f>f_0$,
    \[
    |\mathcal S_{g,f}| \le f^{4g}C^{f}\,.
    \]
\end{corollary}
We are now in a position to prove Lemma~\ref{lemma:surfaces}, using the accurate estimation of Theorem~\ref{thm:gao} for surfaces of type $g$ smaller than a sufficiently large constant $g_0$, and the estimation of Corollary~\ref{cor:bud} to handle surfaces of type $g>g_0.$
\begin{proof}[\emph{\textbf{Proof of Lemma~\ref{lemma:surfaces}}}]
Set $\mu_n:=\frac32\log n/\sqrt{\gamma n}$. First, let $S\in \mathcal S_{0,2k+1,n}\setminus\cT_{k,n}$ be a non-spherical surface filling of $(123)$ which is an \emph{improper} triangulation of $(123)$ with $k$ internal vertices and yet it is face-proper. There are at most $|\cT_k|n^{k-1}k^2$ such surface fillings since at least two of the $k$ vertices have the same label. In addition, since the labeling is face-proper, we infer from~\eqref{eq:poisson-tail} that
\[
\P(w(S)\le\mu_n) = (1+o(1))\frac{\mu_n^{2k+1}}{(2k+1)!}\,.
\]
Therefore, if $\hat Z_{k,n}$ is the number of such triangulations $S$ with $w(S)\leq\mu_n$, and $C_0>0$ is such that $\Delta_k \leq C_0 \gamma^k k^{-5/2}$ for all $k$, then
\begin{align*}
\sum_{k} \E[\hat Z_{k,n}]
&\le(1+o(1)) \sum_{k=3}^{\infty}\Delta_kn^{k-1}k^2\frac{\mu^{2k+1}}{(2k+1)!} \leq (1+o(1))\frac{C_0}{\sqrt \gamma} \sum_{k=3}^{\infty}\frac{1}{\sqrt{k}}\P(\Po(\tfrac32\log n)=2k+1)\,.
\end{align*}
Noting that
\[ \sum_{k=3}^\infty \frac1{\sqrt k} \P(\Po(\tfrac32\log n)=2k+1) \leq \P(\Po(\tfrac32\log n ) \leq \log n) + \frac1{\sqrt{\log n}} = \frac{1+o(1)}{\sqrt{\log n}}\,,\]
we deduce that $\sum_k \E[\hat Z_{k,n}] = O(1/\sqrt{\log n})$.

The main part of the proof is concerned with surface fillings of positive type $g\ge 1/2.$ Let \[Z_{g,f,n}=|\{S\in\mathcal S_{g,f,n}\,:\;w(S)\le\mu_n\}|\,.\]
Using~\eqref{eqn:gfk}, let $k=f/2-1/2-2g$ denote the number of internal vertices in a surface filling from $\mathcal S_{g,f,n}$. As before, we have that
\begin{align*}
\E[Z_{g,f,n}]\leq(1+o(1))|\mathcal S_{g,f}|n^k\frac{\mu_n^{f}}{f!}=(1+o(1))|\mathcal S_{g,f}|\gamma^{-f/2}n^{-1/2-2g}\frac{(\frac32\log n)^{f}}{f!}.
\end{align*}
Let $A=3Ce/\sqrt{\gamma}$ where $C$ is the constant from Corollary~\ref{cor:bud}, and let $g_0$ be a sufficiently large constant we determine later in the proof. We bound the expectation of $Z_{g,f,n}$ by handling the following cases, noting that every summation of $g$ is taken over integers and half-integers.
\begin{enumerate}[\indent{\bf Case (}1{\bf):}]
    \item If $f\leq\sqrt{\log n}$ then we bound $|\mathcal S_{g,f}|\leq \binom{k^3}f< f^{3f},$ where the second inequality follows from $k< f$ (in fact $k<f/2$ by~\eqref{eqn:gfk}). We also use that $f!\ge (f/e)^f$ to find
    \[
    \E[Z_{g,f,n}] \le  \left( \frac{e\cdot \frac32\log n\cdot f^2}{\sqrt{\gamma}} \right)^fn^{-1/2-2g}=n^{-1/2-2g+o(1)}.
    \]
    Therefore,
    \[\sum_{\substack{f\le\sqrt{\log n}\\1/2\le g\le f/4}} \E[Z_{g,f,n}] = n^{-1+o(1)}\,.\]

    \item If $f\geq \sqrt{\log n}$ and $g\le g_0$ then, by Theorem~\ref{thm:gao},
    \[
    \E[Z_{g,f,n}] \le (C_1+o(1))(C_2k)^{5(g-1)/2}{(\gamma n)}^{-1/2-2g}\frac{(\frac32\log n)^f}{f!}.
    \]
    We multiply and divide by $\gamma n$, use that $k<f/2$ and sum over $f$ to obtain
    \begin{align*}
        \sum_{f\geq\sqrt{\log n}} \E[Z_{g,f,n}] =&~
    (C_1+o(1))n^{1-2g}
    \sum_{f\ge\sqrt{\log n}}\P(\Po(\tfrac32\log n)=f)
    \left(\frac{C_2f}{2}\right)^{5(g-1)/2}\\=&~O(n^{1-2g}(\log n)^{5(g-1)/2}).
    \end{align*}
The sum over all integers and half-integers $1/2\le g \le g_0$ is dominated by $g=1/2$ hence
\[\sum_{\substack{f\ge\sqrt{log n}\\1/2\le g\le g_0}}     \E[Z_{g,f,n}] 
= O((\log n)^{-5/4}).
\]

\item If $g>g_0$ and $f\ge A\log n$ then, by Corollary \ref{cor:bud} and $f!>(f/e)^f$ we have that
    \begin{align*}
    \E[Z_{g,f,n}] \le&~ (1+o(1))f^{4g}C^{f}\gamma^{-f/2}n^{-1/2-2g}\left(\frac{\frac32e\log n}{f}\right)^f \\
    \le&~ \frac{(1+o(1))}{\sqrt{n}}\left(\frac{f^2}n\right)^{2g}\left(\frac{A\log n}{2f}\right)^f\\
    =&~ \frac{(1+o(1))}{\sqrt{n}}\left(\frac{A^2(\log n)^2}{4n}\right)^{2g}\left(\frac{A\log n}{2f}\right)^{f-4g}.
    \end{align*}
We have that $f>A\log n$ and $f>4g$ hence a summation of $\left(\frac{A\log n}{2f}\right)^{f-4g}$ over $f$ is bounded by the geometric series $\sum_{i>0}2^{-i}=1$. Consequently, 
\[
\sum_{f>A\log n}\E[Z_{g,f,n}] \le \frac{1+o(1)}{\sqrt{n}}
\left(\frac{A^2(\log n)^2}{4n}\right)^{2g}=n^{-1/2-2g+o(1)}.
\]
Therefore,
$
\sum_{f>A\log n,~g>g_0}\E[Z_{g,f,n}] = n^{-3/2-2g_0+o(1)}.
$
\item If $g>g_0$ and $\sqrt{\log n}<f<A\log n$ then, by Corollary \ref{cor:bud} and $(\frac32\log n)^f/f!<n^{3/2}$,
    \[
    \E[Z_{g,f,n}] \le (1+o(1))f^{4g}C^{2f}\gamma^{-f/2}n^{1-2g} \le (1+o(1)) n \left(\frac{f^2}n\right)^{2g}\left(\frac{A}{2e}\right)^f.
    \]
    We use the fact that $f<A\log n$ to find that
    \[
    \E[Z_{g,f,n}]\le n^{\left(1 -2g(1-o(1))+A(\log (A/2)-1) \right)}.
    \]
    We are free to choose $g_0$ such that $\sum_{f<A\log n,g_0<g<f/4}\E[Z_{g,f,n}] \le n^{-2}$ (say).
\end{enumerate}
Altogether, we see that $\sum_{f,g}\E[Z_{g,f,n}] = O((\log n)^{-5/4}) = o(1)$. Combining this with the above bound $\sum_k\E[\hat Z_{k,n}]=O(1/\sqrt{\log n})=o(1)$ concludes the proof of the lemma. 
\end{proof}

\begin{proof}[\textbf{\emph{Proof of Theorem~\ref{mainthm:nonplanar}}}]
Recall first that every proper triangulation gives rise to an $\F_2$-homological filling (and also to a spherical surface filling) with the same weight.

By Lemma~\ref{lemma:surfaces}, with high probability the minimum weight over all non-spherical surface fillings of $(123)$ is at least $\frac32\log n/\sqrt{\gamma n}$. Condition on this event, denoted $\mathcal{A}$. If $z$
is an inclusion-minimal $\F_2$-homological filling whose corresponding set of faces does not comprise a proper triangulation of $(123)$, then $w(z) \geq \frac32\log n / \sqrt{\gamma n}$ (otherwise $z$ must correspond to a spherical surface filling, hence its set of faces is a proper triangulation).

Next, condition on the above event $\mathcal{A}$ as well as on the the event $D_n \leq (\frac12 \log n + \sqrt{\log n}) / \sqrt{\gamma n} $ (which holds with high probability --- with room to spare --- as established in Section~\ref{sec:proper}).
Let $z$ be an $\F_2$-homological filling  achieving the minimum weight $F_n$. 
Clearly, $z$ must be inclusion-minimal to achieve the global minimum; thus, by the preceding paragraph, its corresponding set of faces must be a proper triangulation of $(123)$, otherwise $w(z) > (3-o(1)) D_n$. Altogether, $F_n=D_n$ (and hence also $D_n=W_n$) with high probability.
\end{proof}

\subsection*{Acknowledgment} E.L.~was supported in part by NSF grant DMS-1812095.

\bibliographystyle{abbrv}
\bibliography{triangs}

\end{document}